\documentclass{amsart}
\usepackage{amssymb}
\usepackage{amsthm}
\newtheorem{prop}{Proposition}
\newtheorem{cons}{Consequence}[section]
\theoremstyle{definition} \newtheorem{defin}{Definition}[section]

\theoremstyle{remark} \newtheorem{rem}{Remark}[section]
\newcommand{\pn}{\par\noindent} \newcommand{\pmn}{\par\medskip\noindent}

\begin{document}
\title{A dynamical system in the space of convex quadrangles}
\author{Yury Kochetkov}
\date{}
\begin{abstract} Let us consider a family $F(\alpha,\beta,\gamma,\delta)$ of
convex quadrangles in the plane with given angles
$\{\alpha,\beta,\gamma,\delta\}$ and with the perimeter $2\pi$.
Such quadrangle $Q\in F(\alpha,\beta,\gamma,\delta)$ can be
considered as a point $(x_1,x_2,x_3,x_4)\in\mathbb{R}^4$, where
$\{x_1,x_2,x_3,x_4\}$ ---  lengths of edges. Then to $F$ a finite
open segment $I\subset\mathbb{R}^4$ is corresponded. A quadrangle
in $F$, that corresponds to the midpoint of $I$ is called a
\emph{balanced quadrangle}. Let $M$ be the set of balanced
quadrangles. The function $f:M\to M$ is defined in the following
way:  angles of the balanced quadrangle $Q'$, $Q'=f(Q)$, are
numerically equal to edges of $Q$. The map $f$ defines a dynamical
system in the space of balanced quadrangles. In this work we study
properties of this system.
\end{abstract}

\email{yukochetkov@hse.ru, yuyukochetkov@gmail.com} \maketitle

\section{Introduction}
\pn In this work we attempt to construct some kind of a duality in
the space of convex quadrangles. Unlike the construction of works
[1] and [2] we, as in the work [3], study the "duality" between
angles and edges. Precisely, given a convex quadrangle of the
perimeter $2\pi$ we can consider a new convex quadrangle of the
perimeter $2\pi$ with angles numerically equal to lengths of edges
of the initial one. However, the new quadrangle is not uniquely
defined. To achieve the uniqueness we will introduce the notion of
a "balanced quadrangle" and in what follows we will work precisely
with them. \pmn Let $ABCD$
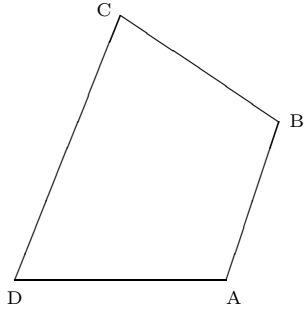
\[\begin{picture}(120,125) \put(10,15){\line(1,0){80}}
\put(10,15){\line(2,5){40}} \put(90,15){\line(1,3){20}}
\put(50,115){\line(3,-2){60}} \put(7,6){\scriptsize D}
\put(90,6){\scriptsize A} \put(114,73){\scriptsize B}
\put(41,115){\scriptsize C} \end{picture}\] be a convex quadrangle
with given angles: $\angle A=\alpha$, $\angle B=\beta$, $\angle
C=\gamma$ and $\angle D=\delta$. Let $x_1$, $x_2$, $x_3$ and
$x_4$, $x_1+x_2+x_3+x_4=2\pi$, be lengths of edges $DA$, $AB$,
$BC$ and $CD$, respectively. To such quadrangle we correspond a
point with coordinates $(x_1,x_2,x_3,x_4)$ in the 4-dimensional
space $\mathbb{R}^4$ and to the family
$F(\alpha,\beta,\gamma,\delta)$ of such quadrangles --- an open
interval $I\subset\mathbb{R}^4$.
\begin{defin} The quadrangle $Q\in F(\alpha,\beta,\gamma, \delta)$
to which the midpoint of $I$ is corresponded is called  a
\emph{balanced quadrangle}.\end{defin} \pn Let $M$ be the set of
balanced quadrangles and let $f$ be the map $f:M\to M$, defined in
the following way.
\begin{defin} The map $f$ corresponds to a balanced quadrangle $Q$
the balanced quadrangle $Q'$ such, that angles of $Q'$ (in the
clockwise order) are numerically equal to the lengths of edges of
$Q$ (also in the clockwise order). \end{defin} \pn Iterations of
the map $f$ have the following properties.
\begin{itemize} \item In the generic case iterations of a quadrangle $Q$
converge to the 2-cycle $(Q_1\leftrightarrow Q_2)$, where balanced
quadrangles $Q_1$ and $Q_2$ are equal up to mirror symmetry
(angles of $Q_1$ see in Section 4). This statement has an
empirical status. \item Two cases of a special behavior are
considered in Section 3: a) the square is a stable repelling
point; b) iterations of a quadrangle with edges
$(\alpha,\beta,\alpha,\gamma)$ converges to an attracting 2-cycle.
\end{itemize}

\section{Explicit formulas}
\pn Let $\alpha,\beta,\gamma,\delta$ --- angles of a convex quadrangle
enumerated in the counter-clockwise order. Consider two pairs of
angles --- $\{\delta,\alpha\}$ and $\{\beta,\gamma\}$. In one pair sum
of angles is less, than $\pi$, in other --- greater, than $\pi$. The
same is true about pairs $\{\alpha,\beta\}$ and $\{\gamma,\delta\}$.
Let $\alpha+\delta<\pi$ and $\gamma+\delta<\pi$. Let us consider the set
$F(\alpha,\beta,\gamma,\delta)$ and the corresponding interval $I$. Two
endpoints of $I$ correspond to triangles --- "degenerate" quadrangles.
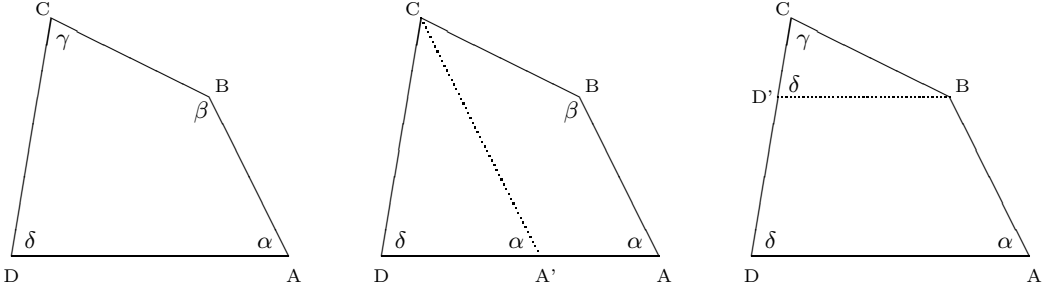
\[\begin{picture}(400,110) \put(5,15){\line(1,0){105}}
\put(5,15){\line(1,6){15}} \put(20,105){\line(2,-1){60}}
\put(80,75){\line(1,-2){30}} \put(10,18){\small $\delta$}
\put(98,18){\small $\alpha$} \put(74,67){\small $\beta$}
\put(22,95){\small $\gamma$} \put(2,5){\scriptsize D}
\put(109,5){\scriptsize A} \put(82,77){\scriptsize B}
\put(14,106){\scriptsize C}

\put(145,15){\line(1,0){105}} \put(145,15){\line(1,6){15}}
\put(160,105){\line(2,-1){60}} \put(220,75){\line(1,-2){30}}
\put(150,18){\small $\delta$} \put(238,18){\small $\alpha$}
\put(214,67){\small $\beta$} \put(193,18){\small $\alpha$}
\put(142,5){\scriptsize D} \put(249,5){\scriptsize A}
\put(222,77){\scriptsize B} \put(154,106){\scriptsize C}
\qbezier[50](160,105)(182,60)(205,15) \put(203,5){\scriptsize A'}

\put(285,15){\line(1,0){105}} \put(285,15){\line(1,6){15}}
\put(300,105){\line(2,-1){60}} \put(360,75){\line(1,-2){30}}
\put(290,18){\small $\delta$} \put(378,18){\small $\alpha$}
\put(299,77){\small $\delta$} \put(302,95){\small $\gamma$}
\put(282,5){\scriptsize D} \put(389,5){\scriptsize A}
\put(362,77){\scriptsize B} \put(294,106){\scriptsize C}
\qbezier[40](295,75)(327,75)(360,75) \put(285,73){\scriptsize D'}
\end{picture}\] In the left figure above a convex quadrangle is presented
with angles $\alpha,\beta,\gamma$ and $\delta$. Lengths of edges
$DA$, $AB$, $BC$ and $CD$ are $x_1$, $x_2$, $x_3$ and $x_4$,
respectively, $x_1+x_2+x_3+x_4=2\pi$. In the middle and right
figures triangles $A'CD$ and $BCD'$ of the perimeter $2\pi$ are
presented --- two results of the "degeneration" of $ABCD$, when
$x_3=0$ in the middle figure and when $x_2=0$ --- in the right.
\pmn Lengths of edges $x_1',x_2',x_3',x_4'$ of the "degenerated
quadrangle" $A'CD$ are:
$$x_1'=\frac{2\pi\cdot\sin(\alpha+\delta)}{\sin(\alpha)+\sin(\delta)+
\sin(\alpha+\delta)}\,,\,x_2'=\frac{2\pi\cdot\sin(\delta)}
{\sin(\alpha)+\sin(\delta)+\sin(\alpha+\delta)}\,,\, x_3'=0,\,
x_4'=\frac{2\pi\cdot\sin(\alpha)}{\sin(\alpha)+\sin(\delta)+
\sin(\alpha+\delta)}\,.$$ Lengths of edges $x_1'',x_2'',x_3'',x_4''$
of the "degenerated quadrangle" $BCD'$ are:
$$x_1''=\frac{2\pi\cdot\sin(\gamma)}{\sin(\gamma)+\sin(\delta)+
\sin(\gamma+\delta)}\,,\,x_2''=0,\,x_3''=\frac{2\pi\cdot\sin(\delta)}
{\sin(\gamma)+\sin(\delta)+\sin(\gamma+\delta)}\,,\,
x_4'=\frac{2\pi\cdot\sin(\gamma+\delta)}{\sin(\gamma)+\sin(\delta)+
\sin(\gamma+\delta)}\,.$$ Now lengths of edges of the balanced quadrangle
with angles $\alpha,\beta,\gamma,\delta$ are:
$$x_1=\frac{x_1'+x_1''}{2}\,,\,x_2=\frac{x_2'+x_2''}{2}\,,\,
x_3=\frac{x_3'+x_3''}{2}\,,\,x_4=\frac{x_4'+x_4''}{2}\,.\eqno(1)$$
\begin{prop} Let $\varphi>0$, $\psi>0$ $\varphi+\psi<\pi$. Then
$$\frac{\sin(\varphi)}{\sin(\varphi)+\sin(\psi)+\sin(\varphi+\psi)}
<\frac
12\quad\text{and}\quad\frac{\sin(\varphi+\psi)}{\sin(\varphi)+
\sin(\psi)+\sin(\varphi+\psi)}<\frac 12\,.$$ \end{prop}
\begin{proof} We have that
$$\frac{\sin(\varphi+\psi)}{\sin(\varphi)+
\sin(\psi)+\sin(\varphi+\psi)}=\dfrac{\cos\frac{\varphi+\psi}{2}}
{\cos\frac{\varphi+\psi}{2}+\cos\frac{\varphi-\psi}{2}}\,.$$ As
$\left|\frac{\varphi-\psi}{2}\right|<\frac{\varphi+\psi}{2}$ then
$$\cos\frac{\varphi-\psi}{2}>\cos\frac{\varphi+\psi}{2}$$
and the second inequality is proved. \pmn Now let us consider
the fraction
$$\frac{\sin(\varphi)}{\sin(\varphi)+\sin(\psi)+\sin(\varphi+\psi)}$$
The derivative of the denominator
$$(\sin(\varphi)+\sin(\psi)+
\sin(\varphi+\psi))'_\psi=2\cdot\cos\frac\varphi2\cdot
\cos\frac{2\psi+\varphi}{2}$$ is positive in the interval $0<\psi<
\frac{\pi-\varphi}{2}$ and is negative in the interval
$\frac{\pi-\varphi}{2}<\psi<\pi-\varphi$. Hence, the fraction is
maximal at points $0$ and $\pi-\varphi$. But here the value of the
fraction is $\frac 12$. \end{proof}
\begin{cons} A balanced quadrangle has two adjacent edges of the
length $\leqslant\frac\pi2$. \end{cons} \begin{proof} Indeed, these are
edges $AB$ and $BC$. \end{proof}  \begin{cons} A convex quadrangle of the
perimeter $2\pi$ with three edges with the lengths $>\frac\pi2$ cannot be
balanced.\end{cons}

\section{The exceptional behavior}
\pn If a balanced quadrangle $Q$ has two equal opposite edges,
then from (1) it follows, that $Q'$ has edges $\{a,a,b,b\}$,
$a<b$, $b=\pi-a$. It means, that $Q''$ is a trapezoid with equal
side edges. The edges of $Q''$ have lengths
$$\left(\frac{\pi}{2+2\cos(a)}\,,\,\frac\pi2\,,\,\frac{\pi}{2+2\cos(a)}\,,\,
\frac\pi2+\frac{\pi\cdot\cos(a)}{1+\cos(a)}\right)\,.$$ Thus,
$Q''$ is of the same type, as $Q$. So, $f$ maps trapezoids to
quadrangles with equal opposite angles and them to trapezoids:
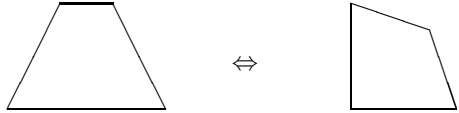
\[\begin{picture}(170,50) \put(0,5){\line(1,0){60}}
\put(0,5){\line(1,2){20}} \put(20,45){\line(1,0){20}}
\put(40,45){\line(1,-2){20}} \put(85,20){$\Leftrightarrow$}
\put(130,5){\line(1,0){40}} \put(130,5){\line(0,1){40}}
\put(130,45){\line(3,-1){30}} \put(160,35){\line(1,-3){10}}
\end{picture}\] A trapezoid with angles $(a,a,b,b)$ after double
iteration of $f$ becomes a trapezoid with angles $(c,c,d,d)$,
$c<d$, where
$$c=\dfrac{\pi}{1+\sin\frac{\pi}{2+2\cos(a)}+\cos\frac{\pi}{2+2\cos(a)}}\,.$$
The convex function $c(a)$, $0<a<\frac\pi2$, increases from the value
$c(0)=\frac{\pi}{\sqrt 2+1}\approx 1.3$ to the value $c(\frac\pi2)=
\frac\pi2$. Its plot in the square $[1.4,\frac\pi2]\times [1.4,\frac\pi2]$
is presented below (the scaling is not preserved):
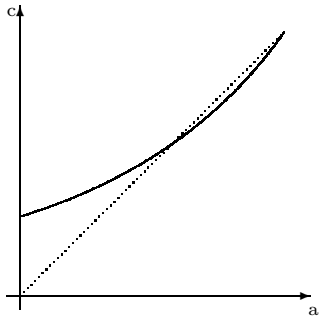
\[\begin{picture}(115,120) \put(0,10){\vector(1,0){115}}
\put(5,5){\vector(0,1){115}} \qbezier[70](5,10)(55,60)(105,110)
\put(114,3){\scriptsize a} \put(0,116){\scriptsize c}
\qbezier(5,40)(70,60)(105,110) \end{picture}\] The curve $c(a)$
and the line $c=a$ intersects in two points: $a\approx
1.48342158769377952440379165224$ and $a=\frac\pi2$. The first point
defines an attracting 2-cycle (the derivative $c'(a)$ at this point
$\approx 0.8$). The point $a=\frac\pi2$ is a repelling stationary point ---
the square. Below are presented angles of quadrangles in the
attracting 2-cycle:
$$(1.48342\ldots,1.48342\ldots,1.65817\ldots,1.65817\ldots)
\leftrightarrow (1.44472\ldots,\frac\pi2,1.44472\ldots,
\frac\pi2+0.25214\ldots)\,$$

\section{The general case}
\pn Computations demonstrate that iterations of a generic balanced
quadrangle converge to an attracting 2-cycle $Q_1\leftrightarrow
Q_2$, where balanced quadrangles $Q_1$ and $Q_2$ are congruent up
to mirror symmetry. Angles $\alpha$, $\gamma$ and $\delta$ of
$Q_1$ are defined by relations
$$\begin{array}{l}
\alpha=\frac{\pi\cdot\sin(\alpha+\delta)}{\sin(\alpha)+
\sin(\delta)+\sin(\alpha+\delta)}+\frac{\pi\cdot\sin(\gamma)}
{\sin(\gamma)+\sin(\delta)+\sin(\gamma+\delta)}\\ \delta=
\frac{\pi\cdot\sin(\delta)}{\sin(\alpha)+\sin(\delta)+
\sin(\alpha+\delta)}\\ \gamma=\frac{\pi\cdot\sin(\delta)}
{\sin(\gamma)+\sin(\delta)+\sin(\gamma+\delta)}\end{array}\eqno(2)$$
and
$$\begin{array}{l}\alpha\approx 1.54819305248669225152933985324\\
\beta\approx 1.82405188512759300508614890573\\
\gamma\approx 1.41515953031350909799654144250\\
\delta\approx 1.49578083925179212231325656509\\
\end{array}\eqno(3)$$ It must be noted that modules of some
derivatives of righthand parts of relations (2) at the point (3)
are greater, than $1$. Thus, the attracting property of this
2-cycle does not have a simple explanation.
\begin{rem} Let generic quadrangles $R$ and $S$ be congruent up to mirror
symmetry. As expected, if $f^{(n)}(R)$ is close to $Q_1$,
$f^{(n)}(S)$ is close to $Q_2$, and vice versa.
\end{rem}

\vspace{1cm}

\begin{thebibliography}{99}
\bibitem{CNS} J. Cantarella, T. Needham, C. Shonkwiler, \emph{Random
triangles and polygons in the plane}, The American Mathematical
Monthly, 126(2), 2019, 113-134. \bibitem{BK} I. Busjatskaja, Y.
Kochetkov, \emph{Dual quadrangles in the plane}, arXiv:
1911.09321. \bibitem{YK} Y.Kochetkov, \emph{Two dynamical systems
in the space of triangles}\,, arXiv: 2101.03734.
\end{thebibliography}
\end{document}